\documentclass[12pt]{amsart}
\usepackage{amssymb,amsmath,amsthm,amsxtra}

\usepackage{colonequals}
\usepackage{fullpage}
\usepackage{hyperref}
\hypersetup{colorlinks=true,urlcolor=blue,citecolor=blue,linkcolor=blue}
\usepackage{mathtools}
\usepackage{comment}
\usepackage[percent]{overpic}
\usepackage{multirow}

\usepackage{xstring} % Needed by Nicolas

\numberwithin{equation}{section}

\newtheorem{thm}[equation]{Theorem}

\newtheorem{proposition}[equation]{Proposition}

\theoremstyle{definition}

\theoremstyle{remark}
\newtheorem{remark}[equation]{Remark}

\usepackage[section]{algorithm}
\usepackage{algpseudocode}
\makeatletter
\let\c@algorithm\c@equation % algorithm counter is exactly the same as equation
\makeatother
\makeatletter
\newenvironment{breakablealgorithm}
  {% \begin{breakablealgorithm}
   \begin{center}
     \refstepcounter{algorithm}% New algorithm
     \vspace{1ex}
     \hrule height.8pt depth0pt \kern2pt% \@fs@pre for \@fs@ruled
     \renewcommand{\caption}[2][\relax]{% Make a new \caption
       {\raggedright\textbf{\ALG@name~\thealgorithm .} ##2\par}%
       \ifx\relax##1\relax % #1 is \relax
         \addcontentsline{loa}{algorithm}{\protect\numberline{\thealgorithm}##2}%
       \else % #1 is not \relax
         \addcontentsline{loa}{algorithm}{\protect\numberline{\thealgorithm}##1}%
       \fi
       \kern2pt\hrule\kern2pt
     }
  }{% \end{breakablealgorithm}
     \kern2pt\hrule\relax% \@fs@post for \@fs@ruled
   \end{center} \vspace{1ex}
  }
\makeatother

\DeclareMathOperator{\GL}{GL}

\newcommand{\psmod}[1]{~(\textup{\text{mod}}~{#1})}

\newenvironment{enumalph}
{\begin{enumerate}}
{\end{enumerate}}

\newenvironment{enumalg}
{\begin{enumerate}}
{\end{enumerate}}

\newenvironment{enumalg2}
{\begin{enumerate}}
{\end{enumerate}}

\numberwithin{equation}{section}

\newcommand{\FF}{\mathbb{F}}
\newcommand{\ZZ}{\mathbb{Z}}
\newcommand{\QQ}{\mathbb{Q}}
\newcommand{\RR}{\mathbb{R}}
\newcommand{\OO}{\mathcal{O}}

\newcommand{\F}{\FF}
\newcommand{\Q}{\QQ}
\newcommand{\Qbar}{\QQ^{\textup{al}}}
\newcommand{\R}{\RR}
\newcommand{\Z}{\ZZ}

\DeclareMathOperator{\Aut}{Aut}
\DeclareMathOperator{\End}{End}
\DeclareMathOperator{\Frob}{Frob}
\DeclareMathOperator{\Gal}{Gal}
\DeclareMathOperator{\Jac}{Jac}
\DeclareMathOperator{\Hom}{Hom}
\DeclareMathOperator{\M}{M}
\DeclareMathOperator{\SL}{SL}
\DeclareMathOperator{\tr}{tr}

\newcommand{\frakl}{\mathfrak{l}}
\newcommand{\frakp}{\mathfrak{p}}

\newcommand{\defi}[1]{\textsf{#1}} % for defined terms

\newcommand{\PP}{\mathbb{P}}
\newcommand{\C}{\mathbb{C}}
\newcommand{\GQ}{\Gal_{\Q}}
\newcommand{\newf}{\mathcal{N}}
\newcommand{\T}{\mathbb{T}}
\newcommand{\g}{\mathfrak{g}}
\newcommand{\p}{\mathfrak{p}}
\newcommand{\eps}{\varepsilon}
\newcommand{\mact}{\big\vert}
\newcommand{\smat}[4]{\left( \begin{smallmatrix} #1 & #2 \\ #3 & #4 \end{smallmatrix} \right)}

\newcommand{\mfurl}[1]{\StrSubstitute{#1}{.}{/}[\mfslash]}
\newcommand{\mfref}[1]{\mfurl{#1}\href{http://www.lmfdb.org/ModularForm/GL2/Q/holomorphic/\mfslash}{\textsf{#1}}}
\begin{document}

\title{A Prym variety with everywhere good reduction over~\texorpdfstring{$\Q(\sqrt{61})$}{Qsqrt61}}

%    Information for first author
\author{Nicolas Mascot}
\address{Department of Mathematics, American University of Beirut, Beirut 1107 2020, Lebanon}
\email{nm116@aub.edu.lb}
\urladdr{\url{https://staff.aub.edu.lb/~nm116/}}

%    Information for second author
\author{Jeroen Sijsling}
\address{Universit\"at Ulm, Institut f\"ur Reine Mathematik, D-89068 Ulm, Germany}
\email{jeroen.sijsling@uni-ulm.de}
\urladdr{\url{https://jrsijsling.eu/}}

\author{John Voight}
\address{Mathematics Department, Dartmouth College, Hanover, NH 03755}
\email{jvoight@gmail.com}
\urladdr{\url{http://www.math.dartmouth.edu/~jvoight/}}

%    General info

% \subjclass[2010]{Primary 11G05; Secondary 14H52}

\date{\today}

% \keywords{Elliptic curves}

\begin{abstract}
  We compute an equation for a modular abelian surface~$A$ that has everywhere good reduction over the quadratic field~$K = \Q(\sqrt{61})$ and that does not admit a principal polarization over~$K$.
\end{abstract}

\maketitle

\section{Introduction} \label{intro}

A foundational theorem of Abrashkin \cite{Abrashkin} and Fontaine \cite{Fontaine} asserts that there is no abelian variety over~$\Q$ with everywhere good reduction; this theorem is an analogue in arithmetic geometry of Minkowski's theorem that every number field has a ramified prime.  More generally, the finiteness of the set of isomorphism classes of abelian varieties over a given number field of given dimension and with good reduction outside a given finite set is a key ingredient in the proof of Faltings \cite{Faltings} that a curve of genus at least two defined over a number field has finitely many rational points.  For this reason and many others, it remains an important task to explicitly study abelian varieties over number fields with controlled ramification \cite{BK,Schoof}.

Moving beyond elliptic curves, Demb\'el\'e--Kumar \cite{DK} have given examples of abelian surfaces with everywhere good reduction over real quadratic fields.  More recently, Demb\'el\'e \cite{Dembele} gave a classification of abelian varieties with everywhere good reduction over~$\Q(\sqrt{d})$ with~$d=53,61,73$: assuming the generalized Riemann hypothesis (GRH), he proves that every such abelian variety is isogenous to the power of a unique abelian surface~$A_d$.  Moreover, for~$d=53,73$ he was able to find an explicit equation for a genus~$2$ curve whose Jacobian is isogenous to~$A_d$.  In this article, we finish off the case~$d=61$ as follows.

\begin{thm} \label{thm:main}
  Let~$K \colonequals \Q(\sqrt{61})$, let~$\nu \colonequals (1+\sqrt{61})/2$, and let~$X$ be the projective plane curve over~$K$ defined by the equation~$F(x,y,z)=0$ where
    \begin{equation} \label{eqn:Fxyz}
      \begin{aligned}
        F(x,y,z) & \colonequals (5\nu + 17)x^4 - (14\nu + 48)x^2y^2 + (8\nu + 28)x^2yz   \\
        &\qquad - (4\nu + 14)x^2z^2 + (10\nu + 33)y^4 - (12\nu + 44)y^3z \\
        &\qquad + (2\nu + 26)y^2z^2 + (4\nu - 16)yz^3 + (\nu - 6)z^4
    \end{aligned}
  \end{equation}
  Then the following statements hold.
  \begin{enumalph}
    \item The curve~$X$ is smooth of genus~$3$ and has bad reduction at the prime~$(2)$ only.
    \item The curve~$X$ admits an involution~$\iota \colon (x: y: z) \mapsto (-x: y: z)$ over~$K$, and the quotient of~$X$ by~$\iota$ defines a map from~$X$ onto an elliptic curve~$E$ of conductor~$(64)$. An equation for~$E$ is given by
       \[ -(10\nu+34)y^2 = x^3 - \nu x^2 + (-\nu+1)x - 1. \]
    \item Up to isogeny over~$K$, we have~$\Jac(X) \sim A \times E$, where~$A$ is the Prym variety of the natural map $\Jac (X) \to \Jac (E)$, an abelian surface over~$K$ with everywhere good reduction.
    \item All geometric endomorphisms of~$A$ are defined over~$K$, and~$\End(A) \simeq \Z[\sqrt{3}]$ so $A$ is of~$\GL_2$-type over~$K$.  Moreover,~$A$ is isogenous to~$A_{61}$ over~$K$.
    \item The abelian surface~$A$ has a polarization over~$K$ of type~$(1,2)$, but no abelian surface isogenous to~$A$ over~$K$ admits a principal polarization over~$K$.
  \end{enumalph}
\end{thm}

The abelian surface~$A_{61}$ arises naturally as an isogeny factor over~$\Q(\sqrt{61})$ of a modular abelian fourfold associated via the Eichler--Shimura construction to the Hecke irreducible space~$S_2(\Gamma_0(61),\chi)$ of classical modular forms of level~$61$ with quadratic character~$\chi$, a space having nontrivial inner twist.  Using the Atkin--Lehner involution and explicit period computations, Demb\'el\'e--Kumar \cite[Remark 3]{DK} observed that~$A_{61}$ is naturally~$(1,2)$-polarized but not principally polarized (using a general criterion of Gonz\'alez--Gu\`ardia--Rotger \cite{GGR}); consequently, this surface was an outlier for their method (the ``sole exception'' in the range of their computations).  
Note that although the full classification results of Demb\'el\'e are conditional on the GRH, the realization in Theorem \ref{thm:main} of the particular abelian surface~$A$ uses a direct computation, making it unconditional.

We construct~$X$ as follows.  First, following Demb\'el\'e's analysis \cite{Dembele}, we consider the modular abelian~$4$-fold~$A_{61}$ attached via the Eichler--Shimura construction to the space~$S_2(\Gamma_0(61), \chi)$.  We compute the period lattice of~$A_{61}$ to large precision.  Adapting the methods of Costa--Mascot--Sijsling--Voight \cite{CMSV}, we then numerically split~$A_{61}$ as the square of a simple abelian surface with a~$(1,2)$-polarization---a \emph{building block} in the language of Ribet and Pyle \cite{Pyle,Ribet-Twists}---providing a numerical candidate for~$A$.  We then glue with the elliptic curve~$E$ (in Theorem \ref{thm:main}(b)) via analytic methods to find a curve of genus three over~$K$, but with a terrible equation.  Finally, we simplify the cover to obtain~$X \to E$ as in Theorem \ref{thm:main} and confirm that the Prym of this cover has everywhere good reduction as claimed.

\subsection*{Contents}
The paper is organized as follows.  In Section \ref{sec:periods}, we explain explicit methods to compute the period matrix of a modular abelian variety to large precision.  In Section \ref{sec:block}, we illustrate our adapted methods to numerically decompose a complex abelian variety with a focus on our example.  In Section \ref{sec:prym}, we realize the resulting building block as a Prym explicitly; then in Sections \ref{sec:simplify} we simplify the equations for the curve~$X$ and verify the properties announced in Theorem \ref{thm:main}.

\subsection*{Acknowledgments}
The authors would like to thank Lassina Demb\'el\'e for comments.  Sijsling was supported by a Juniorprofessurprogramm ``Endomorphismen algebraischer Kurven'' of the Science Ministry of Baden-W\"urttemberg. Voight was supported by a Simons Collaboration grant (550029).  This work was sparked during the conference \emph{Arithmetic of Low-Dimensional Abelian Varieties} at the Institute for Computational and Experimental Research in Mathematics (ICERM) in June 2019; the authors would like to thank ICERM for their hospitality and support.

\section{Computing the period matrix} \label{sec:periods}

In this section, we show how to compute the period matrix of our modular abelian variety to large precision.  We begin by recalling how this variety shows up as a factor in the decomposition up to isogeny of the Jacobian of the modular curve~$X_1(61)$. After this, we explain how to determine the periods of a newform of weight 2 in terms of its~$q$-expansion. Finally, we show how to compute many coefficients of this~$q$-expansion in reasonable time, so as to be able to compute the period matrix with very high accuracy.  (The current implementation in \textsc{Magma} \cite{Magma} allows computations to low precision; this is not enough for our purposes.)

\subsection*{Conventions}
Before we begin, we set up a few conventions we use throughout.  Let~$K$ be a field with algebraic closure~$K^{\textup{al}}$.  For an abelian variety~$A$ over~$K$, we denote by~$\End(A)$ the endomorphisms of~$A$ defined over~$K$; if we wish to consider endomorphisms over a field extension~$L \supseteq K$, we write~$A_L$ for the base change of~$A$ to~$L$ and write~$\End(A_L)$ for endomorphisms over~$L$.

\subsection*{Setup}

Let~$k \in \Z_{\geq 0}$ and~$n,N \in \Z_{>0}$ be integers. We denote by~$\sigma_k(n) \colonequals \sum_{d \mid n} d^k$ the sum of the~$k$-th powers of the positive divisors of~$n$. In particular,~$\sigma_0(n)$ is the number of divisors of~$n$.

Given a subgroup~$H \leq (\Z/N\Z)^\times$, let
\[ \Gamma_H(N) \colonequals \left\{\begin{pmatrix} a & b \\ c & d \end{pmatrix} \in \SL_2(\Z) \, : \, \text{$c \equiv 0 \psmod N$ \ and \ $a,d \bmod N \in H$}\right\}. \]
By a \defi{newform} of level~$\Gamma_H(N)$, we mean a newform of level~$\Gamma_1(N)$ whose nebentypus character~$\eps \colon (\Z/N\Z)^\times \to \C^\times$ satisfies~$H \leq \ker \eps$.  Let~$\newf_k(\Gamma_H(N))$ denote the finite set of newforms of weight~$k$ and level~$\Gamma_H(N)$.  This set is acted on by~$\GQ$, the absolute Galois group of~$\Q$, (on the left) via the coefficients of~$q$-expansions at the cusp~$\infty$: if~$f(q)=\sum_{n=1}^\infty a_n(f) q^n$ lies in~$\newf_k(\Gamma_H(N))$, so that in particular $a_1=1$, then the coefficients~$\{a_n(f)\}_n$ are algebraic integers, and for~$\tau \in \GQ$ we have~$\tau f \in \newf_k(\Gamma_H(N))$ with~$q$-expansion~$(\tau f)(q) = \sum_{n} \tau(a_n(f)) q^n$.
Let~$\GQ\!\backslash \newf_k(\Gamma_H(N))$ be the set of Galois orbits of~$\newf_k(\Gamma_H(N))$. For each~$f \in \newf_k(\Gamma_H(N))$, we denote by~$K_f \colonequals \Q(\{a_n(f)\}_n)$ the number field generated by the coefficients of~$f$.  The field~$K_f$ contains the values of the nebentypus character~$\eps_f$ of~$f$; furthermore the nebentypus of~$\tau(f)$ is~$\eps_{\tau(f)} = \tau(\eps_f)$ for all~$\tau \in G_\Q$.

Let~$\mathcal{H}^* \colonequals \mathcal{H} \cup \PP^1(\Q)$ denote the completed upper half-plane.  Recall that for~$N \in \Z_{>0}$ and~$H \leq (\Z/N\Z)^\times$, the modular curve~$X_H(N) \colonequals \Gamma_H(N) \backslash \mathcal{H}^*$ attached to the congruence subgroup~$\Gamma_H(N)$ is defined over~$\Q$, and its Jacobian~$J_H(N)$ decomposes up to isogeny over~$\Q$ as
\begin{equation}
J_H(N) \sim \prod_{M \mid N} \prod_{f \in \newf_2(\GQ\!\backslash \Gamma_{H_M}(M))} A_f^{\sigma_0(N/M)}, \label{eqn:decomp_mod_jac}
\end{equation}
where~$H_M \leq (\Z/M\Z)^\times$ denotes the image of~$H$ in~$(\Z/M\Z)^\times$.  Moreover, for each Galois orbit~$f \in \GQ\!\backslash \newf_2(\Gamma_H(N))$, the abelian variety~$A_f$ is simple over~$\Q$ with~$\dim A_f = [K_f : \Q]$, and $A_f$ has endomorphism algebra~$\End(A_f)\otimes \Q \simeq K_f$.

Let~$\T$ be the Hecke algebra of weight~$2$ and level~$\Gamma_H(N)$. Roughly speaking,~$A_f$ can be thought of as the piece of~$J_H(N)$ on which~$\T$ acts with the same eigenvalue system as~$f$. More precisely,~$A_f$ is defined by
\begin{equation}
A_f \colonequals J_1(N) / I_f J_1(N),
\end{equation}
where~$I_f \colonequals \{ T \in \T  : Tf = 0\}$ is the annihilator of~$f$ under~$\T$, so that for instance the image of~$T_n \vert_{A_f} \in \End(A_f)$ under the isomorphism~$\End(A_f)\otimes \Q \simeq K_f$ is the coefficient~$a_n(f) \in K_f$. In particular, up to isogeny, the matrix
\[ \left( \begin{array}{ccc} & \vdots & \\ \cdots & \displaystyle \int_{\gamma_j} \tau(f)(z) \, \mathrm{d}z & \cdots \\ & \vdots & \end{array}  \right) \]
whose rows are indexed by the newforms~$\tau(f)$ in the Galois orbit of~$f$ and whose columns are indexed by a~$\Z$-basis~$(\gamma_j)_j$ of
\[ H_1(X_H(N),\Z)[I_f] = \{ \gamma \in H_1(X_H(N),\Z) : \text{$T \gamma = 0$ for all~$T \in I_f$} \} \]
is a period matrix of~$A_f$.

It is therefore possible to extract a period matrix of~$A_f$ out of a period matrix of the modular curve~$X_H(N)$. While it may seem excessive to work with the periods of the whole Jacobian~$J_H(N)$ instead of~$A_f$, this approach presents the advantage of making it easy to compute explicitly in the~$\T$-module~$H_1(X_H(N),\Z)$ using modular symbols.

\subsection*{Our modular Jacobian factor}\label{sect:Decomp_JH}

In what follows, we use LMFDB \cite{LMFDB} labels for our classical modular forms.  In our case, we observe that the Galois orbit of newforms attached to~$A_{61}$ is
\[ f \colonequals \mfref{61.2.b.a} = q+\alpha q^2 + (\alpha^2+3)q^3 + O(q^4) \]
where~$\alpha^4+8\alpha^2+13=0$; the four newforms in this orbit correspond to the four embeddings of~$\Q(\alpha) \hookrightarrow \C$, and their nebentypus is the quadratic character~$\chi_{61}$. Thus~$N=61$ and~$H = \ker \chi_{61}$, the subgroup of squares of~$(\Z/61\Z)^\times$.

This Galois orbit~\mfref{61.2.b.a} happens to be the only one with quadratic nebentypus, whereas there are two Galois orbits with trivial nebentypus, namely
\[ \mfref{61.2.a.a} = q-q^2-2q^3 +O(q^4) \]
whose coefficient field is~$\Q$, and
\[ \mfref{61.2.a.b} = q + \beta q^2 + (-\beta^2+3) q^3 + O(q^4) \]
whose coefficient field is~$\Q(\beta)$ where~$\beta^3-\beta^2-3\beta+1=0$. Since~$61$ is prime and since there are no cuspforms of weight~2 and level~1, the decomposition~\eqref{eqn:decomp_mod_jac} informs us that~$J_H(61)$ is isogenous to the product of an elliptic curve corresponding to~\mfref{61.2.a.a}, an abelian 3-fold corresponding to~\mfref{61.2.a.b}, and our 4-fold~$A_{61}$ corresponding to~\mfref{61.2.b.a}. It also follows that the modular curve~$X_H(61)$ has genus~$8$, which is reasonable for practical computations.

In order to compute the period lattice of~$A_{61}$, we must first isolate the rank-8 sublattice~$H_1(X_H(61),\Z)[I]$ of the rank-16 lattice~$H_1(X_H(61),\Z)$, where~$I \leq \T$ is the annihilator of \mfref{61.2.b.a}. We can obtain this sublattice as
\[ H_1(X_H(61),\Z)[I] = \bigcap_{T \in I} \ker( T \vert_{H_1(X_H(61),\Z)}). \]
We observe that the Hecke operator~$T_2^4+8 T_2^2 + 13$ lies in~$I$, since~$T_2 f = a_2(f) f = \alpha f$ and~$\alpha^4+8\alpha^2+13=0$; furthermore, the kernel of this operator acting on~$H_1(X_H(61),\Z)$ happens to have rank 8, so that it agrees with the sublattice~$H_1(X_H(61),\Z)[I]$.

\begin{remark}
The fact that the hunt for this sublattice was so short-lived is not a coincidence. Indeed, since~$61$ is prime and since there is no cuspform of weight 2 and level 1, all the eigenforms of level 61 are new, so that~$\T \otimes \Q$ is a semi-simple~$\Q$-algebra by the multiplicity-one theorem. It is therefore \'etale, so that most of its elements generate it as a~$\Q$-algebra; in fact, an element of~$\T \otimes \Q$ is a generator if and only if its eigenvalues on each of the Galois conjugates of \mfref{61.2.a.a}, \mfref{61.2.a.b}, and \mfref{61.2.b.a} are all distinct. In particular, it is not surprising to find that~$\T \otimes \Q = \Q[T_2]$, so that~$I \otimes \Q$ is generated by~$T_2^4+8 T_2^2 + 13$, which explains why taking the kernel of this operator led us directly to the correct sublattice.
\end{remark}

\subsection*{Twisted winding elements}

We now recall a formula that we will use to compute the periods~$\int_\gamma f(z) \, \mathrm{d}z$ of a newform~$f \in \newf_2(\Gamma_H(N))$ for~$\gamma$ ranging over a~$\Z$-basis of~$H_1(X_H(N),\Z)$.

Since~$f$ is an eigenform, it is actually enough to compute these integrals for~$\gamma$ ranging over a generating set of~$H_1(X_H(N),\Z)$ as a~$\T$-module, since
\[ \int_{T \gamma} f(z) \, \mathrm{d}z = \int_\gamma (T f)(z) \, \mathrm{d}z = \lambda_T(f) \int_\gamma f(z) \, \mathrm{d}z \]
for all~$T \in \T$, where~$\lambda_T(f)$ is such that~$Tf = \lambda_T(f) f$.

Our strategy consists in integrating term-by-term the~$q$-expansion of~$f$ along modular symbols chosen so that the resulting series has good convergence properties. We choose to use the so-called \emph{twisted winding elements}, that is to say modular symbols of the form
\[ s_\chi = \sum_{a \bmod m} \overline \chi(-a) \{ \infty, a/m \} \]
where~$\chi$ is a primitive Dirichlet character whose modulus~$m$ is coprime to~$N$.

Indeed, let~$\eps_f$ be the nebentypus of~$f$, and~$\lambda_f \in \C^\times$ be the Fricke pseudo-eigenvalue of~$f$, i.e.
\begin{equation}
f \mact \smat{0}{-1}{N}{0} = \lambda_f \sum_{n=1}^{\infty} \overline{a_n(f)} q^n; \label{eqn:WN_eigenval}
\end{equation}
if~$N$ is squarefree, the exact value of~$\lambda_f$ can be read off the coefficients~$a_n(f)$ of~$f$ \cite[Theorem 2.2.1]{Nicolas2}; in any case, a numerical approximation of~$\lambda_f$ can be obtained by evaluating~\eqref{eqn:WN_eigenval} at a particular point of the upper half plane.
A computation \cite[Proposition 3.2.3]{Nicolas2} based on the fact that the twist~$f \otimes \chi \colonequals \sum_{n=1}^{\infty} a_n(f) \chi(n) q^n$ of~$f$ by~$\chi$
is modular of level~$m^2 N$ and on the identity
\[ \int_\infty^0 = \int_\infty^{\frac{i}{m\sqrt{N}}} + \int_{\frac{i}{m\sqrt{N}}}^0 \]
yields the formula
\begin{equation}
\int_{s_\chi} f(z) \, \mathrm{d}z = \frac{m}{2 \pi i} \sum_{n=1}^{\infty} \frac1n \left(\frac{\chi(n) a_n(f)}{\g(\chi)}-\chi(-N) \eps_f(m) \lambda_f \frac{\overline{\chi(n) a_n(f)}}{\g(\overline \chi)}\right) R^n \label{eqn:series_twisted_periods}
\end{equation}
where~$R \colonequals e^{-2\pi/m\sqrt{N}}$ and~$\g(\chi)$ denotes the Gauss sum
\[ \g(\chi) \colonequals \sum_{a \in \Z/m\Z} \chi(a) e^{2 a \pi i / m}  \]
of~$\chi$ (and the same with~$\overline \chi$).

In view of the Deligne bound~$a_n(f) \ll n^{\frac12+\eps}$, it is necessary to sum numerically the series~\eqref{eqn:series_twisted_periods} up to~$n \approx \frac{m \sqrt{N} \log 10}{2 \pi} D$ in order to compute this integral with~$D$ correct decimal digits.
For instance, the homology of the curve~$X_H(61)$ introduced above happens to be generated as a~$\T$-module by the~$s_{\chi}$ for~$\chi$ ranging over the set of primitive characters of modulus~$m \leq 7$; therefore, in order to compute the period lattice of~$A_{61}$ with 1000 digits of accuracy, we need to compute the coefficients~$a_n(f)$ for~$n$ up to approximately~$20000$. We show how this can be done in reasonable time below.

\begin{remark}
For safety, we actually computed the coefficients~$a_n(f)$ and summed the series up to~$n \le 24000$. This~$20\%$ increase is quite arbitrary, but at this point we do not need to have a very tight control over the accuracy of our computations. Note that since we sum the series~\eqref{eqn:series_twisted_periods} using floating-point arithmetic and since its terms are roughly geometrically decreasing in magnitude, summing them in reverse (from~$n=24000$ to~$n=1$) increases numerical stability.
\end{remark}

\subsection*{High-accuracy \texorpdfstring{$q$}{q}-expansions}

Classical algorithms, such as that based on modular symbols or that based on trace formulas for Hecke operators, compute the coefficients~$a_n(f)$ of a newform for~$n \leq B$ in complexity which is at least quadratic in~$B$. Since we are aiming for~$B = 24000$, these algorithms are not suitable for our purpose. We therefore turn to the method presented by Mascot \cite[3.1]{Nicolas1}, which can compute the first~$B$ coefficients~$a_n(f)$ in complexity which is quasi-linear in~$B$.

This method is based on the observation that any two rational functions~$u$ and~$v$ on an algebraic curve~$X$ are necessarily algebraically dependent, since the transcendence degree of the function field of~$X$ is only~$1$. There must therefore exist a polynomial~$\Phi(x,y)$ satisfying~$\Phi(u,v)=0$. Besides, this polynomial can be taken to have degree at most~$\deg v$ in~$x$ and~$\deg u$ in~$y$; indeed, a given value of~$u$ generically corresponds to~$\deg u$ points of~$X$, which in turn yield at most~$\deg u$ values of~$v$.

Let $f \in \newf_k(\Gamma_H(N))$, let~$K_f$ be its coefficient field, and let $\OO_f$ be the ring of integers of~$K_f$. Our method thus consists in using such a relation between the inverse of the~$j$-invariant
\begin{equation} 
u \colonequals 1/j = \frac{E_4^3-E_6^2}{(12 E_4)^3} \in q\Z[[q]]
\end{equation}
where
\[ \quad E_4 = 1+240 \sum_{n=1}^{\infty} \sigma_3(n) q^n, \quad E_6 = 1-504 \sum_{n=1}^{\infty} \sigma_5(n) q^n, \]
and the rational function $v \colonequals f/u' \in \OO_f[[q]]$ where $u' \colonequals q\,\mathrm{d}u/\mathrm{d}q$.

The first~$B$ coefficients of $u$ can be computed in quasi-linear complexity in~$B$ using the sieve of Eratosthenes and fast series arithmetic, and we can then determine the first~$B$ coefficients of~$v$ and hence of $f=vu'$ by Newton iteration using~$\Phi(u,v)=0$.

In order to increase the efficiency of this computation, we run it modulo a large enough prime number~$p$ which splits completely in the coefficient field~$K_f$ of~$f$. More precisely, let~$(\omega_j)_{j \leq [K_f:\Q]}$ be a~$\Z$-basis of~$\OO_f$, so that each element of~$\OO_f$ may be uniquely written as~$\sum_j \lambda_j \omega_j$ with~$\lambda_j \in \Z$ for all~$j$. The matrix
\[ \left( \begin{array}{ccc} & \vdots & \\ \cdots & \tau(\omega_j) & \cdots \\ & \vdots & \end{array}  \right) \]
whose rows (resp. columns) are indexed by the complex embeddings $\tau : K_f \hookrightarrow \C$ (resp. the integers $1 \leq j \leq [K_f:\Q]$) is square, and invertible since the square of its determinant is the discriminant of $K_f$. Let $(c_{j,\tau})_{j,\tau}$ be the coefficients of its inverse, and define
\[ C \colonequals \max_j \sum_{\tau} \vert c_{j,\tau} \vert. \]
Then for all~$r >0$ and for any element~$\theta = \sum_j \lambda_j \omega_j \in \OO_f$, we have
\begin{equation} 
\vert \tau(\theta) \vert \leq r \text{ for all } \tau : K_f \hookrightarrow \C \ \ \Longrightarrow \ \ \vert \lambda_j \vert \leq Cr \text{ for all }  j.
\end{equation}

Since for all~$n \in \Z_{\geq 1}$, the coefficient~$a_n(f) \in \OO_f$ has absolute value at most~$\sigma_0(n) \sqrt{n}$ under every embedding~$\tau : K_f \hookrightarrow \C$, if we require that
\begin{equation}
p \ge 1+ 2 \, C \max_{n \le B} \sigma_0(n) \sqrt{n}, \label{eqn:p_Deligne}
\end{equation}
then for all $n \le B$ we can recover the exact value of $a_n(f) = \sum_j \lambda_{j,n} \omega_j$ from the $\lambda_{j,n} \bmod p$, since we will always have $p \geq 2 \vert \lambda_{j,n} \vert +1$.

\bigskip

Our method is summarized as follows; for more detail, see Mascot \cite[3.1]{Nicolas1}, \cite[3.3]{Nicolas2}.

\begin{breakablealgorithm}
  \caption{High-accuracy~$q$-expansion in quasi-linear time}
  \label{algo:qexp}
  \begin{flushleft}
    \textbf{Input:} A newform~$f$ and~$B \in \Z_{\geq 1}$.

    \textbf{Output:} The coefficients~$a_n(f)$ for~$n \le B$.

    \textbf{Algorithm:}
  \end{flushleft}
  \begin{enumalg}
    \item Evaluate the degrees of the rational functions~$u=1/j$ and~$v=f/u'$ (where $u'=q\,\mathrm{d}q/\mathrm{d}u$) on the modular curve~$X_H(N)$ in terms of the index~$[\operatorname{SL}_2(\Z) : \Gamma_H(N)]$ and of the ramification of the $j$-invariant map.
    \item Use Riemann-Roch to find a bound~$b \in \Z_{\geq 1}$ such that any linear combination of the~$u^m v^n$ with~$m \leq \deg v$ and~$n \leq \deg u$ whose first~$b$~$q$-expansion coefficients vanish is necessarily the 0 function.
    \item Compute the coefficients~$a_n(f)$ for~$n \leq b$ using one of the quadratic algorithms mentioned above.
    \item Pick a prime~$p$ satisfying \eqref{eqn:p_Deligne} and which splits completely in~$K_f$, and compute the first~$B$ coefficients of~$u \bmod p$ using fast series arithmetic in~$\F_p[[q]]$.
    \item For each prime~$\p \mid p$ of~$K_f$, in parallel:
    \begin{enumalg2}
    \item By linear algebra, find coefficients~$a_{m,n} \in \F_\p$ such that the first~$b$ coefficients of \[ \sum_{m \leq \deg v} \sum_{n \leq \deg u} a_{m,n} u^m v^n \]
    are all~$0 \bmod \p$, and define
    \[ \Phi_\p(x,y) = \sum_{m \leq \deg v} \sum_{n \leq \deg u} a_{m,n} x^m y^n \in \F_p[x,y]. \]
    \item Use Newton iteration in~$\F_p[[q]]$ over the second variable of the equation~$\Phi_p(u,v)=0 \bmod \p$ to find the first~$B$ coefficients of~$v \psmod{\p}$.
  \end{enumalg2}
  \item Use CRT over the primes~$\p$ to determine the first~$B$ coefficients of~$v \psmod{p \OO_f}$.
  \item Express the first~$B$ coefficients of~$f=vu' \psmod{p \OO_f}$ on the~$\Z/p\Z$-basis~$(\omega_j)_{j \leq [K_f:\Q]}$ of~$\OO_f/p\OO_f$, and deduce their exact value thanks to Deligne's bounds.
  \end{enumalg}
\end{breakablealgorithm}

\begin{remark}
Although the complexity of this method is quasilinear in~$B$, the initial computation of the~$a_n(f)$ for~$n \leq b$ and the linear algebra required to find the coefficients~$a_{m,n}$ of~$\Phi_\p(x,y)$ represent a significant overhead if the degrees of~$u$ and~$v$ are large. In our case, it is more efficient to proceed as follows:
\begin{enumerate}
\item Compute the first~$B$ coefficients of the newform~$f_0=\mfref{61.2.a.a} = q-q^2-2q^3 +O(q^4)$ introduced above using algorithm~\ref{algo:qexp}. This is faster than for~$f$, since~$f_0$ has trivial nebentypus, so that instead of~$X_H(61)$, we can work on the curve~$X_0(61)$, where the degree of~$u=1/j$ is smaller.
\item Compute the first~$B$ coefficients of~$f$ using a variant of Algorithm~\ref{algo:qexp} based on a relation between the functions~$u=1/j$ and~$v=(f/f_0)^2$. Since the nebentypus of~$f$ has order 2, this can again be done on the curve~$X_0(61)$.
\end{enumerate}
In general, if we wanted the coefficients of more modular forms (e.g., the other newform \mfref{61.2.a.b} of trivial nebentypus), it is better to proceed inductively by replacing the~$j$-invariant with forms that have already been computed (such as~$f_0$ and~$f$), since this leads to polynomials~$\Phi_\p(x,y)$ of much lower degree \cite[3.3]{Nicolas2}.
\end{remark}

Using this method, we can obtain the coefficients~$a_n(f)$ for~$n \leq 24000$ as elements of~$K_f=\Q(\alpha)$, where~$\alpha^4+8\alpha^2+13=0$. We then deduce the complex coefficients of each of the newforms~$\tau(f)$ in the Galois orbit of~$f$ by applying each of the complex embeddings of~$K_f$ to these coefficients.
It then remains to evaluate the integrals~$\int_{s_\chi} \tau(f)(z)\,\mathrm{d}z$ using~\eqref{eqn:series_twisted_periods}, and then to extract the periods of our abelian 4-fold~$A_{61}$, as explained at the beginning of this section.

\section{Isolating the building block} \label{sec:block}

The result of the algorithms in the previous sections is a~$4 \times 8$ period matrix~$\Pi \in \M_{4,8} (\C)$ with respect to some basis of the homology of the corresponding abelian 4-fold~$A_{61}$. In this section, we find the period sublattice corresponding to an abelian surface as isogeny factor.

\subsection*{Finding a polarization}

The homology basis used for the period computations in the previous section need not, in general, be symplectic. Indeed, the computation of the intersection pairing on the homology of a modular curve is currently implemented only for the case of~$X_0(N)$.  Therefore, we start this section by briefly describing how to numerically recover a polarization for~$\Pi$.

\begin{remark}
Demb\'el\'e--Kumar \cite[\S 4.1.1, Method 1]{DK} showed that when the character is quadratic (as in our case), the Atkin--Lehner involution splits the modular abelian variety, giving a polarization whose degree is a power of~$2$.  In cases where the resulting polarization is not principal, they search (in the Hecke algebra) for a suitable principal polarization, so the method we give can be seen as a more general version of this approach.
\end{remark}

More generally, let~$\Pi \in \M_{g,2 g} (\C)$ be a period matrix, with corresponding abelian variety~$A = V / \Lambda$, where~$V = \C^g$ and~$\Lambda = \Pi \Z^{2 g}$. Having~$\Pi$ at our disposal yields an identification of the homology group~$H_1 (A, \Z)$ with~$\Z^{2 g}$.  Multiplication by~$i$ on~$A$ yields an endomorphism~$J$ of~$H_1 (A, \R)$, which we similarly identify with the unique matrix~$J \in \M_{2 g} (\R)$ obtained from the equality \cite[\S 6.1]{BSSVY}
\begin{equation}
  i \Pi = \Pi J .
\end{equation}
Via the above identification and the Riemann relations \cite[4.2.1]{BL}, finding a polarization for~$A$ comes down to finding an alternating matrix~$E \in \M_{2 g} (\Z)$ satisfying~$J^t E J = E$ (where~${}^t$~denotes transpose) and the positive definiteness condition~$i \Pi E^{-1} \Pi^* > 0$ (where~${}^*$~denotes the conjugate transpose).

We first focus on the linear equations
\begin{equation}\label{eq:pollincond}
  \begin{aligned}
    E^t = -E \\
    J^t E J = E .
  \end{aligned}
\end{equation}
Let~$m=g(2g-1)$.  We choose a~$\Z$-basis of size~$m$ of alternating matrices~$\M_{2 g} (\Z)_{\textup{alt}}$ (e.g., indexed by entries above the diagonal).  Numerically solving the equations \eqref{eq:pollincond} then reduces to finding common approximate zeros of a number of linear forms~$\ell_1, \dots, \ell_{m} : \Z^m \to \R$. We determine a~$\Z$-basis of vectors~$v$ for which all~$\ell_i (v)$ are numerically negligible by LLL-reducing the lattice spanned by the rows of the matrix
\begin{equation}
  \left(
  \begin{array}{cccc}
    \multirow{3}{*}{$I_{m}$} &      |   &        &    |     \\
                             & C \ell_1 & \cdots & C \ell_m \\
                             &      |   &        &    |
  \end{array}
  \right) .
\end{equation}
Here~$I_{m}$ is the~$m$-by-$m$ identity matrix, and the other columns correspond to the linear forms~$\ell_i$ multiplied by a sufficiently large scalar~$C \in \R_{\gg 0}$. The first elements of the resulting LLL-reduced basis will give approximate solutions of \eqref{eq:pollincond}, and we preserve those that do so to large enough precision to get a~$\Z$-basis of numerical solutions for \eqref{eq:pollincond}.

Having found a~$\Z$-basis of solutions of \eqref{eq:pollincond}, we take small linear combinations  until we find an element~$E$ of determinant~$1$ such that the Hermitian matrix~$i \Pi E^{-1} \Pi^*$ is positive definite. We summarize our steps in Algorithm \ref{alg:pps}.

\begin{breakablealgorithm}
  \caption{Numerically recovering principal polarizations}
  \label{alg:pps}
  \begin{flushleft}
    \textbf{Input:} A period matrix~$\Pi \in \M_{g,2 g} (\C)$, not necessarily with respect to a symplectic basis.

    \textbf{Output:} An alternating matrix~$E \in \M_{2 g} (\Z)$ defining a principal polarization for the abelian variety corresponding to~$\Pi$.

    \textbf{Algorithm:}
  \end{flushleft}
  \begin{enumalg}
    \item Use LLL to find a~$\Z$-basis~$\left\{ E_1, \dots, E_d \right\}$ of solutions of \eqref{eq:pollincond} in~$\M_{2 g} (\Z)_{\textup{alt}}$.
    \item \label{alg:numer_recov_pol:step2} Take a (random)~$\Z$-linear combination~$E = \sum_{i = 1}^d c_i E_i$, with~$c_i \in \Z$.
    \item If~$\det (E) = 1$ and~$i \Pi E^{-1} \Pi^*$ is positive definite, return~$E$; else, return to Step \ref{alg:numer_recov_pol:step2}.
  \end{enumalg}
\end{breakablealgorithm}

In this way, we quickly recover a principal polarization~$E$ on~$\Pi$.

\begin{remark}
More generally, this method can be used to search for polarizations of any desired type.
\end{remark}

For our modular abelian fourfold~$A=A_{61}$, whose period matrix~$\Pi \in \M_{4,8}(\C)$ was computed in the previous section, we carry out Algorithm \ref{alg:pps} and recover (numerically) a principal polarization on~$A$.  

\subsection*{Finding small idempotents}

With a polarization in hand, we numerically decompose the abelian variety~$A=A_{61}$, following Costa--Mascot--Sijsling--Voight \cite{CMSV}: we compute numerically that~$\End (A_\C)$ is an order of index~$13$ in~$\M_2 (\Z [\sqrt{3}])$.  Using methods similar to those above, we find a~$\Z$-basis~$\left\{ (T_1, R_1), \dots, (T_8, R_8) \right\}$ for the pairs~$(T, R)$ with~$T \in \M_4 (\C)$ and~$R \in \M_8 (\Z)$ such that numerically
\begin{equation}
  T \Pi = \Pi R.
\end{equation}
Specifically,~$R_1=1$, and the remaining matrices~$R_i$ are
\begin{small}
\setlength\arraycolsep{1.5pt}
\begin{gather*}
\begin{pmatrix}
-1 & 0 & 0 & 1 & 1 & 0 & 0 & 0 \\
-1 & -1 & -1 & 0 & 1 & -1 & 2 & 0 \\
1 & 0 & 0 & 0 & -1 & 1 & 0 & 2 \\
-1 & 0 & -1 & 0 & -1 & -1 & -1 & 0 \\
1 & 0 & 1 & 0 & 1 & 1 & 1 & 0 \\
-1 & 0 & -1 & -1 & -1 & -2 & -1 & -2 \\
-1 & 0 & -1 & 0 & -1 & -1 & -1 & 0 \\
1 & 0 & 1 & 0 & 1 & 1 & 1 & 0
\end{pmatrix},
\begin{pmatrix}
-2 & 0 & -1 & 0 & 0 & -2 & -1 & -2 \\
1 & 0 & 0 & 1 & 1 & 0 & 1 & 2 \\
-1 & -1 & -1 & -2 & -1 & 0 & 0 & 0 \\
0 & -1 & -2 & 0 & 0 & 0 & 0 & 0 \\
1 & 1 & 1 & 1 & -1 & 2 & 1 & 2 \\
-1 & -1 & -1 & 0 & -1 & 0 & -1 & 0 \\
-1 & -1 & 0 & 1 & 1 & -2 & 0 & 0 \\
1 & 2 & 2 & 0 & 1 & 1 & 1 & 0
\end{pmatrix},
\begin{pmatrix}
0 & 0 & 0 & 1 & 1 & 0 & 0 & 0 \\
0 & 0 & 0 & 0 & 2 & 0 & 2 & 0 \\
0 & 0 & 0 & 1 & 0 & 0 & 1 & 2 \\
0 & 0 & 0 & 1 & 0 & 0 & 1 & 2 \\
0 & 0 & 0 & -2 & -1 & 0 & -1 & -2 \\
-1 & 0 & -1 & 0 & -1 & -1 & 0 & 0 \\
0 & 0 & 0 & 2 & 0 & 0 & 0 & 2 \\
0 & 0 & 0 & -2 & 0 & 0 & -1 & -3
\end{pmatrix}, \\
\begin{pmatrix}
-1 & 0 & 0 & -1 & 0 & -2 & -1 & -2 \\
0 & 0 & -1 & -2 & -1 & 1 & -2 & -1 \\
0 & -1 & -1 & 1 & 0 & 0 & 1 & 1 \\
-1 & -1 & 0 & 0 & 0 & 0 & 1 & 1 \\
1 & 1 & 0 & -1 & -1 & 0 & -1 & -1 \\
-1 & 0 & 0 & 1 & 0 & 1 & 1 & 2 \\
-1 & -1 & 0 & 1 & 0 & 0 & 0 & 1 \\
1 & 1 & 0 & -1 & 0 & 0 & -1 & -2
\end{pmatrix},
\begin{pmatrix}
-1 & 0 & 0 & 0 & 0 & 0 & 0 & 0 \\
-1 & -1 & -1 & -1 & -2 & -1 & -1 & -1 \\
0 & 0 & -1 & 1 & 0 & 0 & 0 & 1 \\
0 & 0 & 0 & -1 & -1 & 0 & 0 & 1 \\
2 & 0 & 2 & 1 & 2 & 2 & 1 & 1 \\
0 & 0 & 0 & 0 & 1 & -1 & 1 & 0 \\
-2 & 0 & -2 & -2 & -1 & -2 & -1 & -3 \\
1 & 0 & 1 & 1 & 0 & 1 & -1 & 0
\end{pmatrix}, \\
\begin{pmatrix}
1 & 0 & 1 & -1 & 0 & 0 & 0 & 0 \\
2 & 1 & 1 & -2 & 1 & 1 & -2 & -3 \\
-3 & -1 & -2 & -1 & -1 & -2 & -2 & -3 \\
2 & -1 & 0 & -1 & -1 & 2 & 0 & 1 \\
-1 & 1 & -1 & 1 & -1 & 0 & 1 & 1 \\
-1 & -1 & -1 & 0 & -2 & 1 & 0 & 2 \\
1 & -1 & 2 & 1 & 0 & 0 & -1 & 1 \\
-1 & 2 & 0 & 0 & 2 & -1 & 1 & -2
\end{pmatrix},
\begin{pmatrix}
1 & 0 & 0 & 0 & 0 & 0 & 0 & 0 \\
1 & 1 & 0 & -2 & 0 & -2 & -2 & -3 \\
1 & 1 & 1 & 1 & 1 & 1 & 1 & 1 \\
0 & 1 & 1 & 0 & 1 & -1 & 0 & -1 \\
-3 & -1 & -1 & 1 & -1 & 3 & 1 & 3 \\
0 & 0 & 1 & 2 & 1 & 1 & 2 & 2 \\
2 & -1 & 1 & 1 & 1 & -3 & 1 & -1 \\
-1 & 0 & -2 & -2 & -2 & 2 & -2 & 0
\end{pmatrix}.
\end{gather*}
\end{small}

The matrices~$R_i$ are a~$\Z$-basis for the numerical endomorphism ring~$\End(A_\C) \subset \M_8(\Z)$; we look for small idempotents therein as follows.  Inspired by algorithms for computing an isomorphism with the matrix ring \cite{Fisher,IRS}, we equip~$\M_8(\Z)$ with the Frobenius norm (sum of squares of entries)~$N \colon \M_8(\Z) \to \Z$, a positive definite quadratic form; then idempotents (and, more generally, zerodivisors) are likely to be short vectors with respect to this norm.  This very quickly identifies scores of idempotents
\begin{equation}
R_2+1, R_4+1, R_2-R_3+1, R_8, \dots
\end{equation}
and we select one that gives us the simplest looking decomposition.

\subsection*{Inducing the polarization}

Let~$(T, R)$ be a small idempotent in~$\End (A_\C) \otimes \Q$ as constructed above. We can then find an isogeny factor of~$A$ in two ways:
\begin{enumerate}
  \item Taking the kernel of~$T$ yields a linear subspace of~$W$ of~$\C^g$ in which~$W \cap \Pi \Z^{2 g}$ is a lattice. Choosing bases, we get a new period matrix~$\Pi'$ as well as a pair abusively still denoted~$(T, R)$ such that~$T \Pi' = \Pi R$. The matrix~$T$ represents the \emph{inclusion} of~$W \cong \C^{\dim (W)}$ into~$V \cong \C^g$.
  \item Alternatively, taking the image of~$T$ yields a linear subspace of~$W$ of~$\C^g$ in which~$T \Pi \Z^{2 g}$ is a lattice. Choosing bases, we get a new period matrix~$\Pi'$ as well as a pair abusively still denoted~$(T, R)$ such that~$T \Pi = \Pi' R$. The matrix~$T$ represents the \emph{projection} from~$V \cong \C^g$ to~$W \cong \C^{\dim (W)}$.
\end{enumerate}

Note that the dimensions in approaches (i) and (ii) above are complementary in the sense that they sum to~$g$.

\begin{proposition}
  In approach \textup{(i)}, the restriction of~$E$ to~$W$ is represented by the matrix~$E' \colonequals R^t E R$, and~$E'$ defines a polarization for~$\Pi'$.
\end{proposition}

\begin{proof}
  It is clear that~$E'$ is alternating. The complex structure~$J$ restricts to give the complex structure on~$W$, in the sense that because~$T \Pi' = \Pi R$ we have
  \begin{equation}
    \Pi R J' = T \Pi' J' = i T \Pi' = i \Pi R = \Pi J R,
  \end{equation}
  so that~$R J' = J R$. Therefore~$E'$, as the restriction of~$E$ under~$R$, is again compatible with~$J'$, or more precisely
  \begin{equation}
    J'^t E' J' = J'^t R^t E R J' = R^t J^t E J R = R^t E R = E' .
  \end{equation}
  Similarly, since~$E'$ is the restriction of~$E$ under~$R$, the corresponding Hermitian form on~$W$ remains positive definite, which is equivalent to~$i \Pi' E'^{-1} \Pi'^* > 0$ (as in Birkenhake--Lange \cite[Proof of Lemma 4.2.3]{BL}).
\end{proof}

\begin{proposition}\label{prop:caseii}
In approach \textup{(ii)}, an integer multiple of~$E' \colonequals (R E^{-1} R^t)^{-1}$ induces a polarization on~$\Pi'$.
\end{proposition}

\begin{proof}
  This follows from the previous proposition by double duality. The projection~$(T, R)$ induces a dual map~$(T^{\vee}, R^{\vee}) : (A')^{\vee} \to A^{\vee}$ on the corresponding dual abelian varieties. Recall that we can identify~$A^{\vee} = V^{\vee} / \Lambda^{\vee}$, where~$V^{\vee} \colonequals \Hom_{\overline{\C}} (V, \C)$ is the set of antilinear maps from~$V$ to~$\C$ and where~$\Lambda^{\vee}$ is the dual of~$\Lambda$ under the canonical pairing between~$V$ and~$V^{\vee}$. The original pairing~$E$ gives rise to a map~$V \to V^{\vee}$ sending the lattice~$\Lambda$ into its antilinear dual~$\Lambda^*$. In other words, we have
  \begin{equation}
    E (\Lambda \otimes \Q) = \Lambda^{\vee} \otimes \Q .
  \end{equation}
  Therefore a suitable multiple of the inverse map~$E^{-1} : V^{\vee} \to V$ maps~$\Lambda^{\vee}$ into its dual~$\Lambda$. This defines a polarization on~$A^{\vee}$; it is compatible with the complex structure because~$E$ is, and it is positive definite for the same reason. We can transfer this polarization to~$A'^{\vee}$ by the previous proposition, using the inclusion~$(T^{\vee}, R^{\vee})$ dual to the projection~$(T, R)$. Note that~$R^{\vee} = R^t$ by duality. Dualizing a second time, we get our induced polarization on~$A'$ after again taking a suitable integral multiple if needed.
\end{proof}

\begin{remark}
  If the morphism~$(T, R)$ in the previous proposition corresponds to a map of curves of degree~$d$ under the Torelli map, then~$E'$ is in fact~$d$ times the principal polarization on~$\Pi'$ by Birkenhake--Lange \cite[Lemma 12.3.1]{BL}. However, the above notion of an induced polarization applies whether~$(T, R)$ comes from a map of curves or not.
\end{remark}

We apply approach (ii) to the chosen polarization~$E$ and idempotent~$(T, R)$ for our period matrix~$\Pi \in \M_{4,8} (\C)$ and obtain a matrix~$\Pi_2 \in \M_{2,4} (\C)$ with a polarization~$E_2$ of type~$(1, 2)$, which defines an abelian surface~$A_2$. Note that the above approaches are essentially equivalent by duality; we slightly preferred the second as it comes with a projection map $A \to A_2$.

\section{Computing an equation for the cover} \label{sec:prym}

The~$(1, 2)$-polarized abelian surface~$A_2$ defined by~$\Pi_2$ and~$E_2$ in the previous section was obtained from a numerical idempotent of~$\Pi$ defined over the base field~$K = \Q (\sqrt{61})$. Therefore the abelian surface~$A_2$ is also defined over~$K$. However, the polarization~$E_2$ is not principal, so we do not obtain in this way a Jacobian of a genus~$2$ curve over~$K$.  In this section, we show how to realize~$A_2$ as the Prym variety associated to a map~$X \to A_1$ from a genus~$3$ curve to a genus~$1$ curve.

Along the way, we will continue our numerical computations and retain our sensitivity to fields of definition; at the end, we will rigorously certify our construction.

To obtain a principal polarization, one can first try to take a maximal isotropic subgroup of the kernel of the isogeny~$A \to A^{\vee}$ defined by~$E_2$. This kernel is isomorphic to~$\Z / 2 \Z \times \Z / 2 \Z$ and carries a non-trivial Weil pairing; there are therefore~$3$ such maximal isotropic groups. Calculating the corresponding principally polarized quotients of~$A_2$, one obtains that all of them give rise to indecomposable polarizations and therefore correspond to algebraic curves under Torelli. However, calculating the corresponding Igusa invariants, or alternatively reconstructing after Gu\`ardia \cite{Guardia} (implemented at \cite{Sijsling-curve}) shows that all of them define the same cubic Galois extension~$L$ of~$K$, namely that defined by the polynomial~$x^3 - \sqrt{61} x^2 + 4 x + 2$.

This cubic polynomial also shows the way out. We will glue~$A_2$ with an elliptic curve~$A_1$ over~$K$ to realize~$A_2$ as the Prym of genus-$3$ double cover of~$A_1$ over~$K$. This is a slight generalization of the same gluing operation when both~$A_2$ and~$A_1$ have a principal polarization, which will be considered in detail in the upcoming work by Hanselman \cite{Hanselman} (for code, see Hanselman--Sijsling \cite{HS-gluing}).

Via the aforementioned cubic polynomial, we consider the elliptic curve
\begin{equation}
  A_1 \colon y^2 = x^3 - \sqrt{61} x^2 + 4 x + 2 .
\end{equation}
Let~$\Pi_2$ be a period matrix for the~$2$-dimensional abelian variety~$A_2$ with respect to which the polarization~$E_2$ is represented by the alternating form
\begin{equation}
  E_2 =
  \begin{pmatrix}
    0 & -1 & 0 & 0 \\
    1 & 0 & 0 & 0 \\
    0 & 0 & 0 & -2 \\
    0 & 0 & 2 & 0 \\
  \end{pmatrix} .
\end{equation}
The elliptic curve~$A_1$ will have a principal polarization represented by the standard symplectic matrix~$E_1 = \left(\begin{smallmatrix} 0 & -1 \\ 1 & 0 \end{smallmatrix}\right)$, and with corresponding period matrix~$\Pi_1$ say.

Consider the period matrix
\begin{equation}\label{eq:Pp3}
  \Pi'_3 =
  \begin{pmatrix}
    \Pi_2 & 0 \\
    0 & \Pi_1
  \end{pmatrix} ,
\end{equation}
with corresponding abelian variety~$A'_3$; it admits a polarization represented by
\begin{equation}\label{eq:Ep3}
  E'_3 =
  \begin{pmatrix}
    E_2 & 0 \\
    0 & 2 E_1
  \end{pmatrix}
  =
  \begin{pmatrix}
    0 & -1 & 0 & 0 & 0 & 0 \\
    1 & 0 & 0 & 0 & 0 & 0 \\
    0 & 0 & 0 & -2 & 0 & 0 \\
    0 & 0 & 2 & 0 & 0 & 0 \\
    0 & 0 & 0 & 0 & 0 & -2 \\
    0 & 0 & 0 & 0 & 2 & 0 \\
  \end{pmatrix} .
\end{equation}

Let~$G \colonequals \ker(E'_3) < \frac{1}{2} \Lambda' / \Lambda'$, where~$\Lambda' = H_1 (A'_3, \Z) \cong \Pi'_3 \Z^6$. Then~$G$ is equipped with the basis corresponding to the final four columns of~$E'_3$, and under the corresponding identification
\begin{equation}\label{eq:G}
  G \simeq (\Z / 2 \Z)^2 \times (\Z / 2 \Z)^2
\end{equation}
the Weil pairing on~$G$ is the product pairing. By Birkenhake--Lange \cite[2.4.4]{BL}, the quotients of~$A'_3$ on which the polarization~$E'_3$ induces a principal polarization are in bijection with the maximal isotropic subgroups~$H < G$---there are~$15$ such subgroups.

\begin{proposition}
Suppose that~$A_3'$ and~$G$ are defined over~$K$.  Then there exists a maximal isotropic subgroup~$H < G$ that is defined over~$K$. In particular, the quotient~$A'_3 / H$ is defined over~$K$.
\end{proposition}

\begin{proof}
  The first factor in \eqref{eq:G} corresponds to the kernel of the polarization~$E_2$ on~$A_2$. By the discussion at the beginning of this section, we know the structure of this group as a Galois module: the elements of order~$2$ are cyclically permuted by a generator of the Galois extension~$L$ of~$K$ defined by~$x^3 - \sqrt{61} x^2 + 4 x + 2$. This is the very same Galois module structure as that on the second factor in \eqref{eq:G}, which corresponds to~$A_1 [2]$.

  We can therefore obtain a Galois-stable isotropic subgroup~$H < G$ by starting with a pair~$(x_1, x_2)$ with~$x_1 \neq 0$ and~$x_2 \neq 0$ and taking the image under Galois, after which we adjoin~$0$:
  \begin{equation}
    H \colonequals \left\{ (0, 0), (x_1, x_2), (\sigma (x_1), \sigma (x_2)), (\sigma^2 (x_1), \sigma^2 (x_2)) = (x_1 + \sigma (x_1), x_2 + \sigma (x_2)) \right\} .
  \end{equation}
  Since~$\sigma_i (x_1)$ and~$\sigma_j (x_1)$ pair to~$1$ for~$i \neq j$ and similarly for~$x_2$, this subgroup is indeed isotropic, because  the product pairing on~$G$ assumes the values~$0 + 0 = 0$ and~$1 + 1 \equiv 0 \pmod{2}$. By construction,~$H$ is defined over~$K$.

  The second part of the proposition follows immediately.
\end{proof}

Given a maximal isotropic subgroup~$H$ of~$G$, say with generators~$h_1$ and~$h_2$, we take inverse images~$\widetilde{h}_1$ and~$\widetilde{h}_2$ in~$\frac{1}{2} \Lambda'$. This yields a lattice
\begin{equation}
  \Lambda = \Lambda' + \Z \widetilde{h}_1 + \Z \widetilde{h}_2 .
\end{equation}
that contains~$\Lambda'$. Choosing a basis of~$\Lambda$, and keeping the basis for~$\Lambda'$ used heretofore, we find a matrix~$R_3 \in M_6 (\Z)$ such that~$\Lambda R_3 = \Lambda'$. By construction,~$E'_3$ induces a principal polarization on the abelian variety corresponding to the period matrix
\begin{equation}
  \Pi_3 = \Pi'_3 R_3^{-1} .
\end{equation}
Note that we have a projection map corresponding to the pair~$(I_3, R_3)$ in the equality~$I_3 \Pi'_3 = \Pi_3 R_3$. Under this map, the induced principal polarization on~$\Pi_3$ corresponds to the polarization class from Proposition \ref{prop:caseii}.

While we can in principle identify the subgroup~$H$ in terms of the columns of~$E'_3$ by using the Abel--Jacobi map from \cite{MN}, we just tried all~$15$ subgroups, as in Algorithm \ref{alg:goodquot}.

\begin{breakablealgorithm}
  \caption{Finding a principally polarized quotient over~$K$ that is a Jacobian}
  \label{alg:goodquot}
  \begin{flushleft}
    \textbf{Input:} The period matrices~$\Pi_2$ and~$\Pi_1$ for the factors~$A_2$ and~$A_1$ over~$K$, with polarizations~$E_2$ and~$E_1$ of type~$(1, 2)$ and~$(1)$ respectively, along with the associated matrices~$\Pi'_3$ and~$E'_3$ as in \eqref{eq:Pp3} and \eqref{eq:Ep3}.

    \textbf{Output:} A matrix~$\Pi_3 = \Pi'_3 R_3^{-1}$ on which~$E'_3$ induces a principal polarization~$E_3$ such that~$\Pi_3$ numerically corresponds to a genus-$3$ curve~$X$ over~$K$, along with~$E_3$ and the invariants of~$X$.

    \textbf{Algorithm:}
  \end{flushleft}
  \begin{enumalg}
    \item\label{alg:find_ppal_pol_Jac:step1} Determine the~$15$ isotropic subgroups~$H$ of~$G = \ker (E'_3) \cong (\Z / 2 \Z)^2 \times (\Z / 2 \Z)^2$ with the product Weil pairing.
    \item For a given~$H$ as in Step \ref{alg:find_ppal_pol_Jac:step1}, lift generators~$h_1$ and~$h_2$ to elements~$\widetilde{h}_1$ and~$\widetilde{h}_1$ of~$\frac{1}{2} \Lambda' / \Lambda'$.
    \item\label{alg:find_ppal_pol_Jac:step3} Construct~$\Lambda = \Lambda' + \Z \widetilde{h}_1 + \Z \widetilde{h}_2$, choose a basis of~$\Lambda$, and find~$R_3 \in M_6 (\Z)$ such that~$\Lambda R_3 = \Lambda'$.
    \item Let~$\Pi_3 = \Pi'_3 R_3^{-1}$, and let~$E_3 = R_3^{-t} E'_3 R_3^{-1}$.
    \item Using the methods of \cite[\S 2]{Ketal} and LLL, check if the principally polarized abelian variety defined by~$\Pi_3$ and~$E_3$ corresponds to an algebraic curve~$X$, and if so, whether the invariants of~$X$ are numerically in~$K$. If so, return~$\Pi_3$,~$E_3$, and the corresponding invariants. Otherwise proceed to the next~$H$ in Step \ref{alg:find_ppal_pol_Jac:step1}.
  \end{enumalg}
\end{breakablealgorithm}

Invoking this calculation using 1000 decimal digits, only one isotropic subgroup numerically gave rise to invariants over~$K$; we obtain in this way a period matrix~$\Pi_3$ whose abelian variety~$A_3 \colonequals \C^3/\Pi_3\Z^6$ is equipped with the principal polarization~$E_3$.

Finally, using the main result in Lercier--Ritzenthaler--Sijsling \cite{LRS} (recognizing invariants), we constructed a plane quartic equation defining a genus~$3$ curve whose Jacobian (numerically) matches the abelian surface~$A_3$.

\section{Simplifying, twisting, and verifying the cover} \label{sec:simplify}

The methods from the previous section furnish us with an equation~$F (x, y, z) = 0$ defining a smooth plane quartic curve~$X \subset \PP^2$ over~$K$ whose Jacobian putatively contains the abelian surface~$A=A_2$ and the elliptic curve~$E=A_1$ in Theorem \ref{thm:main} as isogeny factors over~$\Qbar$. Since this reconstruction from invariants is only up to~$\Qbar$-isomorphism, this Jacobian may not yet contain the given factors over~$K$.  Moreover, the implementation of these methods does not yet optimize the equations involved when working over proper extensions of~$\Q$; indeed, the first homogeneous defining polynomial~$F$ is of a rather terrible form, whose coefficients are thousands upon thousands of decimal digits long.  In this section, we simplify, twist, and then verify that the cover is as claimed.

\subsection*{Simplification}

Since the ring of integers of~$K$ has class number~$1$, a first step is to remove superfluous primes from the discriminant by generalizing methods of Elsenhans \cite{Elsenhans}. This approach makes the equation somewhat smaller, but its size is still thousands of digits.

Although we constructed~$X$ from a period matrix~$\Pi$, the reconstruction methods involved only uses invariants and does not
furnish the period matrix of~$\Jac(X)$ in terms of~$\Pi$. Fortunately, using integration algorithms due to Neurohr \cite{Neurohr}, we can simply recompute the period matrix of~$\Jac(X)$ from scratch. Then, with the techniques of Bruin--Sijsling--Zotine \cite{BSZ} one can find the numerical geometric automorphism group~$\Aut(X_\C)$ of the curve~$X$ along with its tangent representation on the ambient~$\PP^2$.  It turns out that~$\Aut(X_\C) \simeq \Z / 2 \Z$, and in particular the nontrivial involution~$\iota$ generating~$\Aut(X_\C)$ is putatively defined over~$K$ itself.

Let~$T \in \GL_3 (K)$ be a matrix representing~$\iota$, i.e.,~$F \cdot T = \lambda F$ under the substitution acting on the right.  Since~$X$ is canonically embedded, we can read off the matrix~$T$ from its tangent representation.  Moreover, we can write~$T = U D U^{-1}$, where~$D$ is a diagonal matrix with diagonal entries~$(-\mu, \mu, \mu)$ say. Then for~$H = F \cdot U$ we have~$H \cdot D = F \cdot U D = \lambda F \cdot U = \lambda H$, so that~$H (-x, y, z) = \lambda H (x, y, z)$. In this way, we recognize an exact automorphism and reduce to the case where
\begin{equation}
  F (x, y, z) = G (x^2, y, z) ,
\end{equation}
with~$G$ homogeneous in~$x,y,z$ of weights~$2, 1, 1$. Now the equation~$G (x, y, z) = 0$ defines a genus-$1$ curve~$E$, which is the quotient of~$X$ by its involution~$\iota$ and which is still defined over~$K$. The map~$X \to E$ is ramified over the divisor of zeros~$D$ of~$x$ on~$E$.

We then calculate the~$j$-invariant of the binary quartic~$q (y, z)$ corresponding to the genus-$1$ curve~$E$. An elliptic curve~$E_0$ with the same~$j$-invariant over~$K$ is defined by
\begin{equation}
  E_0 \colon u^2 = p_0(t)
\end{equation}
over~$K$, where
\begin{equation} \label{eqn:p0eqn}
  p_0(t) \colonequals (10\nu+34) (t^3+\nu t^2 + (1-\nu)t+1)
\end{equation}
where~$\nu \colonequals (1+\sqrt{61})/2 \in K$ satisfies~$\nu^2  = \nu + 15$.  The conductor of~$E_0$ is~$(64)$.  As luck would have it, the binary quartic~$q$ is equivalent over~$K$ to the quartic~$z^4 p_0 (y/z)$ corresponding to the polynomial~$p_0$, yielding an isomorphism~$E \xrightarrow{\sim} E_0$. Pushing the divisor~$D$ along this isomorphism, one obtains a divisor~$D_0$ on~$E_0$. Again luck is on our side and the divisor~$D_0 - 4 \infty$ on~$E_0$ is principal; the corresponding degree-$2$ cover~$X_0$ of~$E_0$ is defined by the affine equations
\begin{equation}
  X_0 \colon \left\lbrace
  \begin{array}{ll}
    (10\nu+34)u^2 &\!\!\!\!= t^3+\nu t^2 + (1-\nu)t+1 \\
    (-2\nu+9)w^2 &\!\!\!\!= 2u + (-7\nu - 24)t^2 + (4\nu + 14)t- 2\nu - 7
  \end{array}
  \right.
\end{equation}
Eliminating~$u$ using the second equation and substituting in the first, replacing~$w \leftarrow x$ and~$t \leftarrow y$, and homogenizing yields the defining homogeneous quartic
\begin{equation} \label{eqn:tnuwyuck}
  \begin{aligned}
    &F_0 (x,y,z) = (-195\nu + 859)x^4 + (-64\nu + 282)x^2y^2 + (-5\nu + 24)y^4 \\
    &\qquad\qquad\qquad + (148\nu - 652)x^2yz + (-16\nu + 68)y^3z + (-74\nu + 326)x^2z^2 \\
    &\qquad\qquad\qquad + (96\nu - 422)y^2z^2 + (-148\nu + 652)yz^3 + (-47\nu + 207)z^4.
  \end{aligned}
\end{equation}

We are blessed for one final time because~$F_0$ defines a non-singular curve~$X_0 \subset \PP^2$ with the same Dixmier--Ohno invariants as the original terrible curve~$X$. (This is not automatic, as a degree~$2$ cover of a genus-$1$ curve is not determined by its ramification locus only, in contrast to the case of genus-$0$ curves.) In any event, we have finally obtained a simplified equation~$X_0$.

\subsection*{Twisting}

It remains to find the correct twist of~$X_0$ (with respect to the involution on~$X_0$) that matches~$A_2$ as an isogeny factor \emph{over~$K$} (instead of merely over~$\Qbar$).  Again, this step is needed because the algorithms in Lercier--Ritzenthaler--Sijsling \cite{LRS} only yield the correct geometric isomorphism class.  We proceed as follows.

First, we considering the scheme over~$\Z$ cut out by \eqref{eqn:tnuwyuck} and the equation for~$\nu$.  
By a Gr\"obner basis computation, we find that~$X_0$ has bad reduction only at the prime~$(2)$; so our desired twist will be by a~$2$-unit (in~$K^\times$).  We begin by computing generators for the group~$U \colonequals \Z_K[1/2]^\times/\Z_K[1/2]^{\times 2}$ of~$2$-units of~$K$ modulo squares, a group with~$2$-rank~$3$.

To find the right twist, we loop over primes~$p \geq 5$ that split in~$K$ and by counting points compare the~$L$-polynomial~$L_p(X,T) \in 1+T\Z[T]$ of~$X$ and the absolute~$L$-polynomial of~$f$
\[ L_p(f,T) \colonequals (1-a_pT + \chi(p)pT^2)(1-\tau(a_p)T+\chi(p)T^2) \in 1 + T\Z[T], \]
where~$\langle \tau \rangle = \Gal(K\,|\,\Q)$.  We find always that~$L_p(f, \pm T) \mid L_p(X,T)$ for a unique choice of sign~$\eps_p$.

We then look for an element~$\delta \in U$ such that~$\left(\displaystyle{\frac{\delta}{\p}}\right) = \eps_p$ for each prime~$\p$ above~$p$: we find a unique match, namely~$\delta = -5\nu + 22$, an element of norm~$-1$.  Twisting by this element (replacing~$x^2 \leftarrow \delta^{-1} x^2$) gives an equation all of whose good Euler factors match:
\begin{equation} \label{eqn:Fxyz-madeit}
  \begin{aligned}
    F(x,y,z) & \colonequals (5\nu + 17)x^4 - (14\nu + 48)x^2y^2 + (8\nu + 28)x^2yz   \\
            &\qquad - (4\nu + 14)x^2z^2 + (10\nu + 33)y^4 - (12\nu + 44)y^3z \\
            &\qquad + (2\nu + 26)y^2z^2 + (4\nu - 16)yz^3 + (\nu - 6)z^4.
  \end{aligned}
\end{equation}

\subsection*{Verification}

We now conclude with the proof of our main result (Theorem \ref{thm:main}); for convenience, we reproduce the statement.

\begin{thm}
Let~$X$ be the projective plane curve over~$K$ defined by the equation~$F(x,y,z)=0$ as in \eqref{eqn:Fxyz-madeit}.
  Then the following statements hold.
  \begin{enumalph}
    \item The curve~$X$ is smooth of genus~$3$ and has bad reduction at the prime~$(2)$ only.
    \item The curve~$X$ admits an involution~$\iota \colon (x: y: z) \mapsto (-x: y: z)$ over~$K$, and the quotient of~$X$ by~$\iota$ defines a map from~$X$ onto an elliptic curve~$E$ of conductor~$(64)$. An equation for~$E$ is given by
       \[ -(10\nu+34)y^2 = x^3 - \nu x^2 + (-\nu+1)x - 1. \]
    \item Up to isogeny over~$K$, we have~$\Jac(X) \sim A \times E$, where~$A$ is the Prym variety of the natural map $\Jac (X) \to \Jac (E)$, an abelian surface over~$K$ with everywhere good reduction.
    \item All geometric endomorphisms of~$A$ are defined over~$K$, and~$\End(A) \simeq \Z[\sqrt{3}]$ so $A$ is of~$\GL_2$-type over~$K$.  Moreover,~$A$ is isogenous to~$A_{61}$ over~$K$.
    \item The abelian surface~$A$ has a polarization over~$K$ of type~$(1,2)$, but no abelian surface isogenous to~$A$ over~$K$ admits a principal polarization over~$K$.
  \end{enumalph}
\end{thm}

\begin{proof}
  Part (a) was verified directly via a Gr\"obner basis computation as above: more precisely, we compute that the ideal~$N$ generated by~$F$ and its derivatives in~$\Z[\nu][x,y,z]$, saturated at the irrelevant ideal~$(x,y,z)$, has~$N \cap \Z[\nu] = (2)^{14}$.

  Part (b) is because~$F(x,y,z)=G(x^2,y,z)$, and~$G(x,y,z)$ defines an elliptic curve, equipped with the point at infinity~$(\nu - 3 : 1 : 0)$; we then compute the given equation (arising up to quadratic twist by construction).  For the first part of (c), the splitting over~$K$ comes from the map in (b).

  Next, we prove part (d) using the methods of Costa--Mascot--Sijsling--Voight \cite{CMSV}, which rigorously certifies that~$\End(\Jac(X))=\End(\Jac(X^{\textup{al}})) \simeq \Z \times \Z[\sqrt{3}]$, where~$X^{\textup{al}}$ denotes the base change to an algebraic closure~$K^{\textup{al}}$ of~$K$.  More precisely, with respect to a~$K$-basis of differentials, we certify that there is an endomorphism with tangent representation
  \begin{equation}
    \begin{pmatrix}
      1 & 0 & 0 \\
      0 & 3/2 & (\nu - 4)/2 \\
      0 & (\nu + 3)/2 & -3/2
    \end{pmatrix}
  \end{equation}
  representing multiplication by~$\sqrt{3}$.

  To finish the proof of (c), frustratingly it is not clear how to algorithmically verify that the Prym has everywhere good reduction directly.  (We could not find an implementation for the computation of regular models over number fields.)  Instead, we look at Galois representations and apply the now standard method of Faltings--Serre, the relevant details of which we now sketch: see Dieulefait--Guerberoff--Pacetti \cite[\S 4]{DGP} and more generally Brumer--Pacetti--Poor--Tornar\'ia--Voight--Yuen \cite[\S 2]{BPPTVY}.

  The~$2$-adic Tate module of the abelian surface factor yields a Galois representation
  \begin{equation}
    \rho_{A,2} \colon \Gal_K \to \GL_2(\Z_2[\sqrt{3}]),
  \end{equation}
  where~$\Gal_K \colonequals \Gal(K^{\textup{al}}\,|\,K)$ is the absolute Galois group of~$K$; since~$(2)=\frakl^2$ ramifies in~$\Z[\sqrt{3}]$, reducing modulo~$\frakl=(1+\sqrt{3})$ gives the mod~$\frakl$ representation~
  \[ \overline{\rho}_{A,\frakl} \colon \Gal_K \to \GL_2(\F_2). \]  
  Counting points, we find that the conjugacy classes of Frobenius elements for either prime above~$5$ has order~$3$, so the image of~$\overline{\rho}_{A,\frakl}$ has order~$3$ or~$6$, and correspondingly the fixed field of~$\ker \overline{\rho}_2$ is an extension of~$K$ of degree~$3$ or~$6$ ramified only at~$(2)$.  By computational class field theory, we verify there is in fact a \emph{unique} such extension: it is the normal closure of the extension of~$K$ defined by the polynomial~$p_0(t)$ in \eqref{eqn:p0eqn}, whose absolute field is defined by
  \begin{equation}
    x^6 + x^5 - 7x^4 + 14x^2 - 16x + 4;
  \end{equation}
  the full normal extension~$L$ of~$\Q$ is defined by
  \begin{equation} \label{eqn:fullgalQ}
    x^{12} - 4x^{11} + 10x^{10} - 44x^9 + 174x^8
      - 422x^7 + 618x^6 - 512x^5 + 170x^4 + 56x^3 - 58x^2 + 6x + 9.
  \end{equation}

  We repeat the same calculation for the newform~$f \in S_2(\Gamma_0(61),\chi)$: its base change~$f_K$ to~$K$ has level~$(1)$ and so gives a Galois representation~$\rho_{f_K,\frakl} \colon \Gal_K \to \GL_2(\Z_2^{\textup{al}})$ for which the characteristic polynomial of Frobenius lies in~$\Z_2[\sqrt{3}]$.  Reducing modulo~$\frakl$ and semisimplifying gives~$\overline{\rho}_{f_K,\frakl}^{\textup{ss}} \colon \Gal_K \to \GL_2(\F_2^{\textup{al}})$; since the traces belong to~$\F_2$, up to equivalence the representation takes values in~$\F_2$ \cite[Corollary 1]{Taylor}, and as in the previous paragraph we confirm that it is surjective so absolutely irreducible so equal to its semisimplification: indeed, we must have~$\overline{\rho}_{A,\frakl} \simeq \overline{\rho}_{f_K,\frakl}$ by uniqueness.  By a result of Carayol \cite[Th\'eor\`eme 2]{Carayol}, up to equivalence the representation takes values in $\Z_2[\sqrt{3}]$, i.e., 
we have~$\rho_{f_K,\frakl} \colon \Gal_K \to \GL_2(\Z_2[\sqrt{3}])$.

  We conclude by showing that~$\rho_{A,\frakl} \simeq \rho_{f_K,\frakl}$, from which the remainder of part (c) follows as~$f_K$ has level~$(1)$.  We compute~$\Z_L[1/2]^{\times}/\Z_L[1/2]^{\times 2} \simeq (\Z/2\Z)^{9}$ and that the primes of~$L$ above~$3,5,61,97$ generate the corresponding elementary~$2$-abelian extension of~$L$.  The method of Faltings--Serre then provides that~$\rho_{A,\frakl} \simeq \rho_{f_K,\frakl}$ if and only if we have the equality of traces~$\tr \rho_{A,\frakl}(\Frob_\frakp)=\tr \rho_{f_K,\frakl}(\Frob_\frakp)$ for the primes~$\frakp$ of $K$ above $3,5,61,97$, where $\Frob_\frakp \in \Gal_K$ denotes (a representative of) the conjugacy class of the $\frakp$-Frobenius automorphism.  Having checked equality of~$L$-polynomials for these primes (indeed, we checked all primes of norm up to~$200$), the result follows.

  Finally, having proven in the previous paragraph that~$A$ is modular, part (e) was proven for a particular abelian surface in the isogeny class of~$A$ by Demb\'el\'e--Kumar \cite[Remark 3]{DK} (using a criterion of Gonz\'alez--Gu\`ardia--Rotger \cite{GGR}).  We now show it extends to all abelian surfaces in the isogeny class.  We have exhibited a~$(1,2)$-polarization~$\lambda \colon A \to A\spcheck$, where~$A\spcheck$ denotes the dual abelian variety.  
  Assume for purposes of contradiction that there exists an isogeny~$\varphi \colon A' \to A$ and a \emph{principal} polarization~$\lambda'$ on~$A'$, both over~$K$.  Then~
  \[ \mu' \colonequals \varphi\spcheck \circ \lambda \circ \varphi \colon A' \to (A')\spcheck \]
  is another polarization, with~$\deg(\mu') = 2\deg(\varphi)^2$.  From work of Gonz\'alez--Gu\`ardia--Rotger \cite[Theorem 2.3, Theorem 2.10(ii)]{GGR}, we conclude that~$\mu'$ corresponds to a \emph{totally positive} element of~$\End(A')$ of norm~$\deg(\mu')$.  But~$\End(A')$ is an order in~$\End(A') \otimes \Q = \Q(\sqrt{3})$, and~$2$ is not a norm from this field---equivalently, the conic~$x^2-3y^2=2z^2$ has no rational points over~$\Q$, or also equivalently, the corresponding quaternion algebra~$(2,3\,|\,\Q) \not\simeq \M_2(\Q)$ is not split, having discriminant~$6$---a contradiction.  (The element~$1+\sqrt{3} \in \Z[\sqrt{3}]$ has norm~$-2$, and~$(2)$ represents the nontrivial element of the narrow class group of~$\Z[\sqrt{3}]$.)
\end{proof}

\end{document}